\documentclass[12pt,leqno]{amsart}
\usepackage{amsmath,amsfonts,amssymb,amsthm}
\usepackage{graphicx}
\graphicspath{ }
\usepackage{wrapfig}
\usepackage{accents}
\usepackage{caption}
\usepackage{subcaption}
\usepackage{calligra}

\numberwithin{figure}{section}

\theoremstyle{plain}
\newtheorem{thm}{Theorem}[section]

\newtheorem{prop}[thm]{Proposition}
\newtheorem{cor}{Corollary}[thm]
\theoremstyle{definition}

\theoremstyle{remark}

\usepackage{mathtools}
\title[Yamabe soliton and positive scalar curvature]{Non-existence of Riemannian metric satisfying Yamabe soliton}
\author[A. A. Shaikh, C. K. Mondal ]{Absos Ali Shaikh$^1$, Chandan Kumar Mondal$^2$ }
\address{\noindent\newline $^{1}$Department of Mathematics,\newline University of
Burdwan, Golapbag,\newline Burdwan-713104,\newline West Bengal, India}
\email{aask2003@yahoo.co.in, aashaikh@math.buruniv.ac.in}
\address{\noindent\newline $^{2}$School of sciences,\newline Netaji Subhas Open University, Durgapur-713214, \newline West Bengal, India}
\email{chan.alge@gmail.com}

\begin{document}
\begin{abstract}
In this paper we have proved that a compact Riemannian manifold does not admit a metric with positive scalar curvature if there exists a real valued function in this manifold which is strictly positive along a geodesic ray satisfying expanding or steady Yamabe soliton. We have also deduced a relation between scalar curvature and surface area of a geodesic ball in a Riemannian manifold with a pole satisfying steady Yamabe soliton.
\end{abstract}
\noindent\footnotetext{$\mathbf{2010}$\hspace{5pt}Mathematics\; Subject\; Classification: 53C20; 53C21; 53C44. }
\maketitle
\section{Introduction and preliminaries}
One of the most central problem of Riemannian geometry is that: given any smooth real valued function on the manifold is there any Riemannian metric whose associated scalar curvature is that function? This is a very broad question and motivated an enormous amount of researchers over the years. In case of compact manifold, Kazdan and Warner \cite{KW75} proved the existence of such a Riemannian metric when the function is negative. A good amount of research have been done by several authors when the manifold admits a positive scalar curvature, see \cite{GL80,ST92} and in case of negative scalar curvature, see \cite{BK89,LO92}. The conformal version of the above question is that: given a Riemannian manifold with Riemannian metric $g$ and a smooth real valued function on the manifold, is there any Riemannian metric conformal to $g$ whose scalar curvature is that function? When the manifold is compact and given function is negative or zero, Aubin \cite{AU70} showed that there exists such a metric and when the function is positive somewhere then Berger \cite{BE71} has given a solution.
\par The notion of Yamabe flow has been introduced by R. Hamilton \cite{HA88} in which the metric is evolved accordingly as
$$\frac{\partial}{\partial t}g(t)=-R(t)g(t),$$
where $R(t)$ is the scalar curvature of then metric $g(t)$. A connected Riemannian manifold $(M,g)$ with $n\geq 2$ is called Yamabe soliton if there exists a vector field such that
\begin{equation}\label{eq4}
\frac{1}{2}\mathsterling_Xg=(R-\rho)g,
\end{equation}
where $\mathsterling_X$ denotes the Lie derivative in the direction of $X$ and $\rho$ is a scalar. The Yamabe soliton is called expanding, steady and contracting if $\rho<0$, $\rho=0$ and $\rho>0$ respectively. If there exists a function $\varphi\in C^\infty(M)$ such that $X=grad \varphi$, then it is called gradient Yamabe soliton and then equation (\ref{eq4}) takes the form
\begin{equation}\label{eq5}
\frac{1}{2}\nabla^2 \varphi=(R-\rho)g,
\end{equation}
where $\nabla^2$ is the Hessian operator. The concept of pole in a Riemannian manifold was introduced by Gromoll and Meyer \cite{GM69}. A point $p\in M$ of a Riemannian manifold is called a pole \cite{GM69} if the tangent space $T_pM$ at the point $p$ is diffeomorphic to $M$. A smooth function $\varphi:M\rightarrow\mathbb{R}$ is said to be convex \cite{CK19, 8} if for any $p\in M$ and for any vector $v\in T_pM$
$$\left\langle grad \varphi,v\right\rangle_p\leq \varphi(exp_pv)-\varphi(p).$$
\par In section 2, we have proved that if there exists a strictly positive function along a geodesic ray in a compact Riemannian manifold, satisfying expanding or steady Yamabe soliton, then the manifold does not admit any Riemannian metric with positive scalar curvature. In the last section, we have deduced a relation between scalar curvature and surface area of geodesic ball in a complete Riemannian manifold with a pole satisfying steady gradient Yamabe soliton.
\section{Riemannian metric and Yamabe soliton}
The main result of this paper is the following:
\begin{thm}\label{th3}
Let $M$ be a compact Riemannian manifold of dimension $n\geq 3$ and $f\in C^\infty(M)$. Suppose there exists $\delta\in\mathbb{R}$ such that $f\geq\delta>0$ along some geodesic ray satisfying the equation
\begin{equation}\label{eq1}
\frac{1}{2}\mathsterling_Xg=(f-\rho)g, \text{ for }\rho\leq 0.
\end{equation}
Then $M$ does not admit any Riemannian metric with positive scalar curvature.
\end{thm}
\begin{thm}\cite{KW75}\label{th2}
Let $M$ be a compact Riemannian manifold of dimension $n\geq 3$. If $f\in C^\infty(M)$ is negative somewhere, then there is a $C^\infty$ Riemannian metric on $M$ with $f$ as its scalar curvature.
\end{thm}
\begin{thm}\cite{KW75}\label{th1}
Let $M$ be a compact Riemannian manifold of dimension $n\geq 3$. If $M$ admits a metric whose scalar curvature is positive, then any $f\in C^\infty(M)$ is the scalar curvature of some Riemannian metric on $M$.
\end{thm}
\begin{proof}[Proof of Theorem \ref{th3}:]
Let $\gamma:[0,\infty)\rightarrow M$ be a geodesic ray emanating from $p\in M$ and parametrized by arc length $t$ such that $f\geq\delta>0$ along $\gamma$. Suppose $M$ admits a Riemannian metric with positive scalar curvature. Then from Theorem \ref{th1}, there exists a Riemannian metric $g$ with $f$ as its scalar curvature $R$ satisfying the equation (\ref{eq1}) for $\rho\leq 0$. Then along $\gamma$ we have
$$\frac{1}{2}\mathsterling_Xg(\gamma',\gamma')=g(\nabla_{\gamma'}X,\gamma')=\frac{d}{dt}[g(X,\gamma')].$$
Thus from (\ref{eq1}) the above equation implies
$$\int_{0}^{t_0}(R-\rho)_{\gamma}g(\gamma',\gamma')dt=\int_{0}^{t_0}\frac{d}{dt}[g(X,\gamma')]dt=g(X_{\gamma(t_0)},\gamma'(t_0))-g(X_p,\gamma'(0)).$$
Now using the Cauchy-Schwarz inequality, we get
$$\int_{0}^{t_0}R_{\gamma}dt\leq \|X_{\gamma(t_0)}\|-g(X_p,\gamma'(0))+\rho t_0.$$
Since $M$ is compact, there exists a scalar $K$ such that $\|X\|\leq K$. Hence, the above equation yields
$$\int_{0}^{t_0}R_{\gamma}dt\leq K-g(X_p,\gamma'(0))+\rho t_0.$$
Again taking limit on both side, we obtain
$$\lim_{t_0\rightarrow\infty}\int_{0}^{t_0}R_{\gamma}dt\leq K-g(X_p,\gamma'(0))+\lim_{t_0\rightarrow\infty}\rho t_0.$$
Now $\rho\leq 0$ implies that $ \lim_{t_0\rightarrow\infty}\rho t_0\leq 0$. Thus, we have 
\begin{equation}\label{eq2}
\lim_{t_0\rightarrow\infty}\int_{0}^{t_0}R_{\gamma}dt\leq C(p),
\end{equation}
where $C(p)$ is a constant depends on $p$. But from the given condition, we get
$$\lim_{t_0\rightarrow\infty}\int_{0}^{t_0}R_{\gamma}dt\geq \lim_{t_0\rightarrow\infty}\int_{0}^{t_0}\delta dt=+\infty,$$
which contradicts the equation (\ref{eq2}). Therefore, $M$ does not admit any Riemannian metric with positive scalar curvature.
\end{proof}
\begin{cor}
Let $(M,g)$ be a compact Riemannian manifold of dimension $n\geq 3$ with the Riemannian metric $g$ whose scalar curvature is positive. Then there doest not exist any function $f\in C^\infty(M)$ such that $f\geq \delta>0$ on $M$ and $f$ satisfies (\ref{eq1}).
\end{cor}

\begin{prop}
Let $M$ be a compact Riemannian manifold of dimension $n\geq 3$. If $f\in C^\infty(M)$ is negative somewhere and satisfies the equation (\ref{eq1}), then $f$ is non-positive along each geodesic ray.
\end{prop}
\begin{proof}
Since $f$ is negative somewhere, by the Theorem \ref{th2}, there is a Riemannian metric whose scalar curvature $R$ is $f$. Again $f$ satisfies the equation (\ref{eq1}), so by using (\ref{eq2}), the result easily follows.
\end{proof}
\begin{thm}\cite{HSU12}
Let $(M,g)$ be an $n$-dimensional compact Riemannian manifold satisfying Yamabe gradient soliton with $n\geq 3$. Then $(M,g)$ is of constant scalar curvature.
\end{thm}
From the above Theorem, we can state the following result:
\begin{prop}
Let $M$ be a compact $n$-dimensional Riemannian manifold with $n\geq 3$ and $f\in C^\infty(M)$. If there exists a Riemannian metric $g$ whose scalar curvature is $f$ and which satisfies the gradient Yamabe soliton (\ref{eq5}), then $f$ must be constant.
\end{prop}
\section{Upper bound of scalar curvature}
In this section we have given a relation between the scalar curvature and volume and surface area of a geodesic ball. Let $M$ be a Riemannian manifold with a pole $p$. Then by the notation $x\rightarrow\infty$ we mean the limit $d(p,x)\rightarrow\infty$ for $x\in M$.
\begin{thm}\cite{MM12}\label{th5}
Let $(M,g)$ be a complete and non-compact gradient Yamabe soliton with $\rho\geq 0$ with scalar curvature $R$ and $\lim_{x\rightarrow\infty}R(x)\geq 0$. Then the scalar curvature $R$ of $(M,g)$ is non-negative.
\end{thm}
\begin{thm}\cite{19}\label{th4}
Let $B_R(p)$ be a geodesic ball in a complete Riemannian manifold $M$. Suppose that the sectional curvature $K_M\leq k$ for some constant $k$ and $R<inj(M,g)$. Then for any real valued smooth function $f$ with $\Delta f\geq 0$ and $f\geq 0$ on $M$,
\begin{equation}
f(p)\leq \frac{1}{V_k(R)}\int_{B_R(p)}f dV,
\end{equation}
where $V_k(R)=\text{Vol}(B_R,g_k)$ is the volume of a ball of radius $R$ in the space form of constant curvature $k$, $dV$ is the volume form and $inj(M,g)$ is the injective radius of $M$.
\end{thm}
\begin{thm}
Let $(M,p)$ be a complete non-compact Riemannian manifold with a pole $p$ and the sectional curvature $K_M\leq k$ for some constant $k$ with scalar curvature $R$ satisfying $\Delta R\geq 0$. If the Riemannian metric $g$  satisfies the steady gradient Yamabe soliton
\begin{equation}\label{eq3}
\frac{1}{2}\nabla^2 \varphi=Rg,
\end{equation}
where $\varphi\in C^2(M)$ is a convex function, then
$$2R(p)\leq \varphi(x'')\frac{Vol(\partial B_p(R'))}{Vol_k(B_p(R))}-\varphi(x_0)\frac{Vol(\partial B_p(R))}{Vol_k(B_p(R))},$$
where $R<R'$, $x''\in Vol(\partial B_p(R'))$, $x_0\in Vol(\partial B_p(R))$  and $B_R(p)$ is a geodesic ball with center at $p$ and radius $R$.
\end{thm}
\begin{proof}
Taking trace in both sides of (\ref{eq3}), we get
$$\frac{1}{2}\Delta \varphi=Rn.$$
Now taking $\eta$ as the unit outward normal vector along $\partial B_p(R)$ and integrating in $B_p(R)$ an then using divergence theorem and applying convexity property of $\varphi$, we have
\begin{eqnarray*}
\int_{B_p(R)}R dV&=&\frac{1}{2}\int_{B_p(R)}\Delta \varphi dV\\
&=&\frac{1}{2}\int_{\partial B_p(R)}\left\langle X,\eta \right\rangle dS\\
&\leq & \frac{1}{2}\int_{\partial B_p(R)}\{\varphi(exp_x\eta)-\varphi(x)\}dS\\
&\leq & \frac{1}{2}\{\varphi(exp_{x'}\eta)-\varphi(x_0)\}Vol(\partial B_p(R)),
\end{eqnarray*}
where $\varphi(x_0)=\inf\{\varphi(x):x\in\partial B_p(R)\}$ and $\varphi(exp_{x'}\eta)=\sup\{\varphi(exp_x\eta):x\in\partial B_p(R)\}$.
Since $\eta$ is the unit outward normal along $\partial B_p(R)$, the set $\{exp_x\eta:x\in\partial B_p(R)\}=\partial B_p(R')$ for some $R'>R$.
Also, $Vol(\partial B_p(R))\leq Vol(\partial B_p(R'))$. So the above equation reduces to
$$2\int_{B_p(R)}R dV\leq \varphi(x'')Vol(\partial B_p(R'))-\varphi(x_0)Vol(\partial B_p(R)),$$ where $x''=exp_{x'}\eta$. Since $\lim_{x\rightarrow\infty} R(x)\geq 0$, it implies that $R\geq 0$ by using Theorem \ref{th5}. Now using Theorem \ref{th4}, we get
$$2Vol_k(B_p(R))R(p)\leq  \varphi(x'')Vol(\partial B_p(R'))-\varphi(x_0)Vol(\partial B_p(R)).$$
Rearranging the above inequality, we get 
\begin{equation}
2R(p)\leq \varphi(x'')\frac{Vol(\partial B_p(R'))}{Vol_k(B_p(R))}-\varphi(x_0)\frac{Vol(\partial B_p(R))}{Vol_k(B_p(R))}.
\end{equation}
Hence, we get the required inequality.
\end{proof}

\end{document}